\documentclass[11pt]{amsart}
\usepackage{amssymb}
\usepackage{verbatim}
\usepackage{hyperref}
\usepackage{color}

\usepackage{enumitem}

\begin{document}
\setcounter{tocdepth}{1}

\newtheorem*{acknowledgement}{Acknowledgements}
\newtheorem{theorem}{Theorem}[section]
\newtheorem{proposition}[theorem]{Proposition}
\newtheorem{conjecture}[theorem]{Conjecture}
\def\theconjecture{\unskip}
\newtheorem{corollary}[theorem]{Corollary}
\newtheorem{lemma}[theorem]{Lemma}
\newtheorem{sublemma}[theorem]{Sublemma}
\newtheorem{fact}[theorem]{Fact}
\newtheorem{observation}[theorem]{Observation}
\newtheorem{definition}[theorem]{Definition}
\newtheorem{notation}[theorem]{Notation}
\newtheorem{remark}[theorem]{Remark}
\newtheorem{question}[theorem]{Question}
\newtheorem{questions}[theorem]{Questions}

\newtheorem{example}[theorem]{Example}
\newtheorem{problem}[theorem]{Problem}
\newtheorem{exercise}[theorem]{Exercise}

\newcommand{\cA}{\mathcal{A}}
\newcommand{\cB}{\mathcal{B}}
\newcommand{\cC}{\mathcal{C}}
\newcommand{\cD}{\mathcal{D}}
\newcommand{\cE}{\mathcal{E}}
\newcommand{\cG}{\mathcal{G}}
\newcommand{\cH}{\mathcal{H}}
\newcommand{\cI}{\mathcal{I}}
\newcommand{\cJ}{\mathcal{J}}
\newcommand{\cK}{\mathcal{K}}
\newcommand{\cL}{\mathcal{L}}
\newcommand{\cM}{\mathcal{M}}
\newcommand{\cN}{\mathcal{N}}
\newcommand{\cO}{\mathcal{O}}
\newcommand{\cS}{\mathcal{S}}
\newcommand{\cT}{\mathcal{T}}
\newcommand{\cU}{\mathcal{U}}
\newcommand{\cV}{\mathcal{V}}
\newcommand{\cW}{\mathcal{W}}
\newcommand{\cX}{\mathcal{X}}
\newcommand{\cY}{\mathcal{Y}}
\newcommand{\cZ}{\mathcal{Z}}
\newcommand{\bA}{\mathbb{A}}
\newcommand{\bB}{\mathbb{B}}
\newcommand{\bC}{\mathbb{C}}
\newcommand{\bD}{\mathbb{D}}
\newcommand{\bF}{\mathbb{F}}
\newcommand{\bG}{\mathbb{G}}
\newcommand{\bH}{\mathbb{H}}
\newcommand{\bI}{\mathbb{I}}
\newcommand{\bJ}{\mathbb{J}}
\newcommand{\bK}{\mathbb{K}}
\newcommand{\bL}{\mathbb{L}}
\newcommand{\bM}{\mathbb{M}}
\newcommand{\bN}{\mathbb{N}}
\newcommand{\bO}{\mathbb{O}}
\newcommand{\bP}{\mathbb{P}}
\newcommand{\bQ}{\mathbb{Q}}
\newcommand{\bR}{\mathbb{R}}
\newcommand{\bS}{\mathbb{S}}
\newcommand{\bT}{\mathbb{T}}
\newcommand{\bU}{\mathbb{U}}
\newcommand{\bV}{\mathbb{V}}
\newcommand{\bW}{\mathbb{W}}
\newcommand{\bX}{\mathbb{X}}
\newcommand{\bY}{\mathbb{Y}}
\newcommand{\bZ}{\mathbb{Z}}

\newcommand{\R}{\mathbb R}
\newcommand{\Q}{\mathbb Q}
\newcommand{\Z}{\mathbb Z}
\newcommand{\C}{\mathbb C}
\newcommand{\N}{\mathbb N}
\newcommand{\T}{\mathbb T}
\newcommand{\F}{\mathcal F}
\newcommand{\A}{\tilde{A}_{4,d}}
\newcommand{\Aq}{\tilde{A}_{q,d}}
\newcommand{\B}{\mathbb B}
\renewcommand{\S}{\mathcal S}
\newcommand{\w}{\omega}
\newcommand{\e}{\varepsilon}
\newcommand{\g}{\gamma}
\newcommand{\p}{\varphi}
\newcommand{\s}{\psi}
\newcommand{\z}{\zeta}
\renewcommand{\l}{\ell}
\renewcommand{\a}{\alpha}
\renewcommand{\b}{\beta}
\renewcommand{\k}{\kappa}
\newcommand{\m}{\textfrak{m}}
\renewcommand{\P}{\textfrak{p}}

\newcommand{\op}[1]{\operatorname{{#1}}}
\newcommand{\im}{\operatorname{Im}}
\newcommand{\pd}[2]{\frac{\partial #1}{\partial #2}}
\newcommand{\rd}[2]{\frac{d #1}{d #2}}
\newcommand{\ip}[2]{\langle #1 , #2 \rangle}
\renewcommand{\j}[1]{\langle #1 \rangle}
\newcommand{\cl}{\operatorname{cl}}
\newcommand{\into}{\hookrightarrow}

\newcommand{\mc}{\mathcal}
\newcommand{\mb}{\mathbb}
\newcommand{\f}{\mathfrak}
\newcommand{\ssp}{\sqsubseteq}
\newcommand{\cc}{\overline}
\renewcommand{\Re}{\text{Re}}
\renewcommand{\Im}{\text{Im}}


\def\bE{{\mathbf E}}
\def\bedagger{{\mathbf e}^\dagger}
\def\bF{{\mathbf F}}
\def\reals{\mathbb{R}}
\def\be{{\mathbf e}}
\def\symdif{\,\Delta\,}
\def\bH{{\mathbf H}}
\def\bV{{\mathbf V}}
\def\bG{{\mathbf G}}
\def\Psharp{P^\sharp}
\def\bA{{\mathbf A}}
\def\bI{{\mathbf I}}
\def\bEstar{{\mathbf E}^\star}
\def\bAstar{{\mathbf A}^\star}
\def\be{{\mathbf e}}
\def\bv{{\mathbf v}}
\def\bw{{\mathbf w}}
\def\br{{\mathbf r}}
\def\Star{\star}
\def\doublestar{\dagger\ddagger}
\def\unitQ{{\mathbf Q}}
\def\Gl{\operatorname{Gl}}
\def\eps{\varepsilon}
\def\naturals{{\mathbb N}}
\def\rplus{{\mathbb R}^+}
\def\scriptt{{\mathcal T}}
\def\integers{{\mathbb Z}}
\def\one{{\mathbf 1}}

\def\repair{\medskip\hrule\hrule\medskip}

\title{Special cases of power decay in multilinear oscillatory integrals}
\author{Dong Dong}
\address{Center for Scientific Computation and Mathematical Modeling\\
University of Maryland\\
College Park, MD 20742, USA}
\email{ddong12@cscamm.umd.edu}
\author{Dominique Maldague}
\address{Department of Mathematics\\
University of California, Berkeley\\
Berkeley, CA 94720, USA}
\email{dmal@math.berkeley.edu}
\author{Dominick Villano}
\address{Department of Mathematics\\
University of Pennsylvania\\
Philadelphia, PA 19104, USA}
\email{dvillano@sas.upenn.edu}

\date{\today}

\let\thefootnote\relax
\let\thefootnote\relax
\maketitle

\begin{abstract}
We use tools from the multilinear oscillatory integral program developed by Christ, Li, Thiele, and Tao to treat special
cases which are not covered by existing theory. Consideration of special cases leads to an extended class of examples for which $\lambda$-power decay holds. 
\end{abstract}

\section{Introduction} 
Multilinear oscillatory integrals play an important role in analysis \cite{BCCT,C,cltt, CO, FL, GGX, G, GX, X}. In a seminal paper \cite{cltt}, Christ, Li, Tao, and Thiele (CLTT) obtained decay estimates for the following operator:
\begin{equation}
    (f_1,f_2,\dots,f_n)\mapsto \int_{\R^m} e^{i\lambda P(x)}\prod_{j=1}^nf_j(\pi_j(x))\eta(x)\,dx.
\end{equation}
Here $\lambda\in\R$, $P: \R^m\to\R$ is a real-valued polynomial, $m\ge 2$, $\eta$ is compactly supported, and each $\pi_j$ is the orthogonal projection from $\R^m$ to a subspace $V_j$ of dimension $\kappa<m$, which is assumed to be independent of $j$. 

To state CLTT's results, we first recall some definitions from \cite{cltt}.

\begin{definition}
A polynomial $P$ is called degenerate with respect to the projections $\left\{\pi_j \right\}$ if there exist polynomials $\left\{ p_j \right\}$ such that $P = \sum_j p_j \circ \pi_j$.  If no such polynomials exist, $P$ is said to be nondegenerate.
\end{definition}

For a degree $d\ge 1$, fix a norm $\|\cdot\|_d$ on the finite-dimensional space of real-valued polynomials $p:\R^m\to\R$ of degree at most $d$.  
\begin{definition} 
    A collection of polynomials $P_{\alpha}$ and orthogonal projections $\pi_j^{\alpha}$ is uniformly nondegenerate if there exists a positive constant c such that
    $$
        \inf_{\alpha}\inf_{p_j} \|P_\a- \sum p_j \circ \pi_j^\a \|_d \geq c >0
    $$
where the infimum is taken over real-valued polynomials $p_j$ of degree at most $d$. 
\end{definition}

\begin{definition} A collection $\{V_j:1\le j\le n\}$ of $\kappa$-dimensional subspaces of $\R^m$ is said to be in general position if any subcollection of cardinality $k\ge 1$ spans a subspace of dimension $\min(k\kappa,m)$. 
\end{definition}

CLTT observe that nondegeneracy is necessary for the power decay property.  In general, sufficiency remains open.  However, many special cases have been handled, including the following theorem:

\begin{theorem} [Theorem 2.1, \cite{cltt}] \label{maincltt}
    Suppose $n < 2m$ and $\{V_j: 1\le j\le n\}$ is a family of one-dimensional subspaces of $\R^m$ that lies in general position. Then
    \begin{equation} \label{eq: L2 thm}
        \left|\int_{\R^m} e^{i\lambda P(x)}\prod_{j=1}^nf_j(\pi_j(x))\eta(x)\,dx\right|\le C(1+|\lambda|)^{-\epsilon}\prod_{j=1}^n \|f_j\|_{L^2}
    \end{equation}
    for all polynomials $P$ of bounded degree which are uniformly nondegenerate with respect to $\{V_j\}$, for all functions $f_j\in L^2(\R)$, with uniform constants $C>0, \epsilon>0$.
\end{theorem}

We refer to this type of inequality, with a negative power of $\lambda$ in the upper bound, as a $\lambda$-decay result. The proof of Theorem \ref{maincltt} proceeds by slicing the ambient space according to the subspaces and then employing notions of uniformity originating in the work of Gowers \cite{Gow}.  

By the example stated after Theorem 2.1 in \cite{cltt}, the condition $n<2m$ is necessary for the right-hand side of \eqref{eq: L2 thm} to be a product of $L^2$ norms of the $f_j$. Note that because every function $f_j$ may be taken to have compact support, $L^2$ norms are stronger than $L^\infty$ norms. It is interesting to investigate if the condition $n<2m$ can be relaxed, provided we replace some of the $L^2$ norms on the right-hand side of \eqref{eq: L2 thm} by $L^\infty$ norms. More precisely, one can ask
\begin{question} \label{conj}For $n\ge 2m$ and a nondegenerate phase $P$, does $\lambda$-decay  hold with some combination of $L^\infty$ and $L^2$ norms on the right-hand side of \eqref{eq: L2 thm}? 
\end{question}
In this paper, we answer the above question in the affirmative in some special cases. Although our approaches may be presented in greater generality, we will only discuss a model operator for the clearest presentation. 

Consider the model functional    
\begin{align*}  
\Lambda_\lambda(\vec{f})&=\int_{\R^3} e^{i\lambda P(x,y,z)}f_1(y+z)f_2(y-z)f_3(x+z)f_4(x-z)\\
&\qquad\cdot f_5(x+y+\sqrt{2}z)f_6(x+y-\sqrt{2}z)\eta(x,y,z)dxdydz\\
&=:\int_{\R^3}e^{i\lambda P(x)}\prod_{j=1}^6f_j(v_j\cdot(x,y,z))\eta(x,y,z)dxdydz \end{align*} 
where the vectors $v_j\in\R^3$ are defined by $v_j\cdot(x,y,z)$ is the argument of $f_j$ above.
For certain polynomial phases $P(x,y,z)$ described in the following theorem, a grouping technique leads to $\lambda$-decay. 
\begin{theorem} \label{main2} Suppose that the polynomial phase function $P:\R^3\to\R$ satisfies 
\begin{align} \label{hyp} \inf_{z\in \R}\sup_{|(x,y)|\le1}|\partial_x\partial_y(\partial_x-\partial_y)P(x,y,z)|>0. \end{align}
Then there exist $C>0$ and $\epsilon>0$ such that
\[|\Lambda_\lambda(\vec{f})|\le C(1+|\lambda|)^{-\epsilon}\|f_1\|_2\|f_2\|_\infty\|f_3\|_2\|f_4\|_\infty\|f_5\|_2\|f_6\|_\infty \]
for all $\lambda\in \R$ and all $f_1,f_3,f_5\in L^2$ and $f_2,f_4,f_6\in L^\infty$. 
\end{theorem}

When the polynomial phase $P$ depends only on one variable, say $x$, nondegeneracy conditions can be easily checked and we have the following result.

\begin{theorem}\label{x} Let $\Lambda_\lambda$ and the collection of $v_j$ be as above. If $P:\R\to\R$ is a polynomial of degree at least $3$ which is nondegenerate with respect to the projections $(x,y,z)\mapsto v_j\cdot(x,y,z)$ for $j=1,\ldots,6$, then there exist $C>0$ and $\eta>0$ such that
\[ |\Lambda_\lambda(\vec{f})|\le C(1+|\lambda|)^{-\epsilon}\|f_1\|_2\|f_2\|_2\|f_3\|_\infty\|f_4\|_\infty\|f_5\|_\infty\|f_6\|_\infty \]
for all $\lambda\in\R$ and all $f_1,f_2\in L^2$ and $f_3,f_4,f_5,f_6\in L^\infty$. 
\end{theorem}
It is clear from the proof that the inequality in Theorem \ref{x} holds if we switch the index $2$ with any of $3,4,5,6$. The proofs of Theorem \ref{main2} and Theorem \ref{x} are contained in Section \ref{section: grouping} and Section \ref{section: P(x)} respectively. Some examples will also be provided, which demonstrate slight generalizations of each technique as well as comparisons in the approaches. 

This material is based upon work supported by the National Science Foundation under Grant Number DMS 1641020. The first author was partially supported by LTS grant DO 0052. The second author was supported by an NSF graduate research fellowship under Grant No. DGE 1106400.

\section{A grouping technique } \label{section: grouping}
We prove Theorem \ref{main2} using the $L^2$ theorem (stated here as Theorem \ref{maincltt}) of CLTT and a grouping trick. A discussion of examples follows.

\begin{proof}[Proof of Theorem \ref{main2}] 
For each $z\in\R$, group the $f_j$ by defining $F_j^z:\R\to\R$ as
\begin{align*}
    F_1^z(y)&:=f_1(y+z)f_2(y-z)\\
    F_2^z(x)&:=f_3(x+z)f_4(x-z)\\
    F_3^z(x+y)&:=f_5(x+y+\sqrt{2}z)f_6(x+y-\sqrt{2}z).
\end{align*}
Rewrite $\Lambda_\lambda(\vec{f})$ using the $F_j^z$ and $P^z(x,y)=P(x,y,z)$ as
\begin{align} 
\Lambda_\lambda(\vec{f})&=\int_{\R^3}e^{iP^z(x,y)}F_1^z(y)F_2^z(x)F_3^z(x+y)\eta(x,y,z)dydxdz \\
\label{grouped} &= \int_\R\left(\int_{\R^2}e^{iP^z(x,y)}F_1^z(y)F_2^z(x)F_3^z(x+y)\eta(x,y,z)dydx\right) dz.\end{align} 
For each $z$, the hypothesis (\ref{hyp}) implies that $P(x,y,z)$ (as a function of $(x,y)$ with $z$ fixed) is nondegenerate with respect to the projections $(x,y)\mapsto x$, $(x,y)\mapsto  y$, and $(x,y)\mapsto x+y$. Furthermore, the nondegeneracy is uniform in $z$. Thus we may apply the $L^2$ theorem from CLTT to obtain the bound
\[ \left|\int_{\R^2}e^{iP^z(x,y)}F_1^z(x)F_2^z(y)F_3^z(x+y)\eta(x,y,z)dydx\right|\le C(z)(1+|\lambda|)^{-\epsilon}\|F_1^z\|_2\|F_2^z\|_2\|F_3^z\|_2,\]
where $\epsilon>0$ is independent of $z$ and $C(z)<\infty$ depends on the dimension, the degree of the polynomial $P$, the nondegeneracy of $P^z$ (specifically the quantity on the left-hand  side of (\ref{hyp})), and on the uniform norms of some partial derivatives $\eta$. Since the constant $C$ depends continuously on these parameters, and the nondegeneracy is uniform in $z$ and the uniform norms of partial derivatives of $\eta$ do not depend on $z$, $C(z)$ is bounded for all $z$ by a constant $C<\infty$. Also note that for each $z$, by the definition of the $F_j^z$, $\|F_1^z\|_2\le \|f_1\|_2\|f_2\|_\infty$, $\|F_2^z\|_2\le \|f_3\|_2\|f_4\|_\infty$, and $\|F_3^z\|_2\le\|f_5\|_2\|f_6\|_2$. Let $S\subset\R$ be a measurable subset of finite Lebesgue measure which contains the set $\{z\in\R:\eta(x,y,z)\not=0\quad\text{for some}\quad(x,y)\in\R^2\}$. Putting the above discussion together with the expression from (\ref{grouped}) yields the desired bound
\begin{align*}
    |\Lambda_\lambda(\vec{f})|&\le \int_{\R}\left|\int_{\R^2}e^{i\lambda P^z(x,y)}F_1^z(x)F_2^z(y)F_3^z(x+y)\eta(x,y,z)dydx\right|dz\\
    &= \int_S\left|\int_{\R^2}e^{i\lambda P^z(x,y)}F_1^z(x)F_2^z(y)F_3^z(x+y)\eta(x,y,z)dydx \right|dz\\
    &\le |S|C(1+|\lambda|)^{-\epsilon}\|f_1\|_2\|f_2\|_\infty\|f_3\|_2\|f_4\|_\infty\|f_5\|_2\|f_6\|_\infty.
\end{align*}
This finishes the proof of Theorem \ref{main2}.
\end{proof}
\bigskip

Now we turn to examples of phases for which Theorem \ref{main2} applies, but not CLTT's original theorem or Theorem \ref{x}. Example \ref{flex1} shows a grouping technique in which more than one variable must be fixed. Example \ref{flex2} describes some flexibility in the grouping technique, where there are multiple ways to group terms, leading to $\lambda$-decay results with different $L^p$ norms.

Example \ref{1not2} discuss $\lambda$-decay for functionals of the form
\begin{align}
\label{basic}    \Lambda_\lambda(P,\vec{f})&= \int_{\R^3}e^{i\lambda P(x,y,z)}f_1(y+z)f_2(y-z)f_3(x+z)f_4(x-z)\\
\nonumber    &\qquad\qquad\cdot f_5(x+y+\sqrt{2}z)f_6(x+y-\sqrt{2}z)\eta(x,y,z)dyxdz .
\end{align}
We also discuss this type example in examples \ref{2not1}, \ref{both} in the next section, after presenting the proof of Theorem \ref{x}. This is an example of a $6$-linear functional with projections from $\R^3$ to vectors on the light cone $\{(x,y,z):\,\, x^2+y^2=z^2\}$. If the phase $P(x,y,z)$ is simply nondegenerate, then Theorem 2.3 of \cite{cltt} gives $\lambda$-decay with $L^\infty$ norms, i.e.  
\[ |\Lambda_\lambda(P\vec{f})|\le C|\lambda|^{-\epsilon}\prod_{j=1}^6\|f_j\|_\infty. \]
We will henceforth refer to this result as the $L^\infty$ theorem of CLTT. Regardless of the (simple) nondegeneracy condition on $P(x,y,z)$, no theorem from \cite{cltt} gives any mixture of $L^2$ and $L^\infty$ bounds since their $L^2$ theorem (Theorem \ref{maincltt}) applies only for strictly fewer than $2\cdot 3$ factors.

\begin{example}\label{1not2} $P(x,y,z)=x^2y+2xyz$: Theorem \ref{main2} applies, but not Theorem \ref{x}. \end{example}  
\noindent The hypothesis 
\[ \inf_{z\in\R}\sup_{|(x,y)|\le1}|\partial_x\partial_y(\partial_x-\partial_y)(x^2y+2xyz)|=2>0 \]
of Theorem \ref{main2} is satisfied. Thus there exists $C>0$ and $\epsilon>0$ such that 
\[ |\Lambda_\lambda(x^2y+2xyz,\vec{f})|\le C(1+|\lambda|)^{-\epsilon}\|f_1\|_2\|f_2\|_\infty\|f_3\|_2\|f_4\|_\infty\|f_5\|_2\|f_6\|_\infty \]
for all $\lambda\in\R$ and $f_1,f_3,f_5\in L^2$ and $f_2,f_4,f_6\in L^\infty$. Note that the proof of Theorem \ref{x} cannot be repeated for this phase because for each $\z$, there exist polynomials $p_j:\R\to\R$ such that
\[P(x,y,z)-P(x+\z,y,z) = p_2(y-z)+p_3(x+z)+p_4(x-z)+p_5(x+y+\sqrt{2}z)+p_6(x+y-\sqrt{2}z) . \]
This is because 
\begin{align*} 
P(x,y,z)-P(x+\z,y,z)&=x^2y+2xyz-(x+\z)^2y-2(x+\z)yz\\
&= -2\z xy-\z^2y-2\z yz
\end{align*} 
and we can write
\[4xy+4yz=-2(y-z)^2-(x+z)^2-(x-z)^2+(x+y+\sqrt{2}z)^2+(x+y-\sqrt{2}z)^2 \]
\[\text{and}\qquad y=y-z+\frac{1}{2}(x+z)-\frac{1}{2}(x-z).\]

Examples \ref{flex1} and \ref{flex2} below describe grouping techniques for functionals with more factors and integrated over $\R^4$. 

\begin{example}\label{flex1} A grouping approach fixing 2 variables. \end{example} \noindent Consider the functional    
\begin{align*}  
&\Lambda_{2,\lambda}(\vec{f}):=\int_{\R^4} e^{i\lambda x^2y}f_1(y+z)f_2(y-z)f_3(x+z)f_4(x-z)f_5(x+y+\sqrt{2}z)\\
&\qquad\qquad\cdot f_6(x+y-\sqrt{2}z)f_7(x+w)f_8(x-w)f_9(x-z+2w) \eta(x,y,z,w)dxdydzdw. \end{align*} 


Theorems giving $\lambda$-decay from CLTT do not apply because there are too many factors ($9\not<2\cdot4$) for the $L^2$ theorem to apply, and the phase $x^2y$ is annihilated by the differential operator from the definition of simple nondegeneracy, so the $L^\infty$ theorem does not apply. Attempting to fix individual variables separately does not lead to $\lambda$-decay using an argument analogous to the proof of Theorem \ref{main2} for the following reasons.

\begin{enumerate}
    \item Fix w. Then 
    \begin{align*} 
    \Lambda_{2,\lambda}(\vec{f})&=\int_{\R}\left(\int_{\R^3}e^{i\lambda x^2y}f_1(y+z)f_2(y-z)f_3(x+z)F_4^z(x-z)f_5(x+y+\sqrt{2}z)\right.\\
    &\cdot \qquad \qquad \left.f_6(x+y-\sqrt{2}z)F_7^z(x)\eta(x,y,z,w)dydxdz\right)dw \end{align*}
    where $F_4^z(x)=f_2(x)f_9(x+2w)$ and $F_7^w(x)=f_7(x+w)f_8(x-w)$. The integrand in parentheses has a product of $7$ functions, and since $7>2\cdot3$, we cannot use Theorem \ref{maincltt} to get an $L^2$ result. Since the phase $x^2y$ is also not simply nondegenerate with respect to $x,y,z$, there are no other theorem which apply. 
    \item Fix z. Then 
    \begin{align*} 
    \Lambda_{2,\lambda}(\vec{f})&=\int_{\R}\left(\int_{\R^3}e^{i\lambda x^2y}F_1^z(y)F_2^z(x)F_3^z(x+y)f_7(x+w)f_8(x-w)\right.\\
    &\cdot \qquad\qquad \left. F_9^z(x+2w)\eta(x,y,z,w)dydxdz\right)dw \end{align*}
    where $F_1^z(y)=f_1(y+z)f_2(y-z)$, $F_2^z(x)=f_3(x+z)f_4(x-z)$, $F_3^z(x)=f_5(x+\sqrt{2}z)f_6(x-\sqrt{2}z)$, and $F_9^z(x)=f_9(x-z)$. This leads to a product of six factors, and since $6\not>6$, we cannot use Theorem \ref{maincltt} to get an $L^2$ result. Since the phase $x^2y$ is also not simply nondegenerate with respect to $x,y,z$, the $L^\infty$ theorem also does not apply. 
    \item Fixing $x$ does not work because for each $x$, the phase $x^2y$ is degenerate with respect to the grouped projections $(y,z,w)\mapsto y+z$ and $(y,z,w)\mapsto z$. Fixing $y$ does not work for an analogous reason. 
\end{enumerate} 
Now fix $z$ and $w$. Then  
\begin{align*} 
    \Lambda_{2,\lambda}(\vec{f})&=\int_{\R^2}\left(\int_{\R^2}e^{i\lambda x^2y}F_1(x,z,w)F_2(y,z,w)F_3(x+y,z,w)\eta(x,y,z,w)dydx\right)dzdw 
\end{align*}
where 
\[ F_1(x,z,w)=f_3(x+z)f_4(x-z)f_7(x+w)f_8(x-w)f_9(x-z+2w), \]
\[ F_2(y,z,w)=f_1(y+z)f_2(y-z),\quad \text{and} \quad F_3(x+y,z,w)=f_5(x+y+\sqrt{2}z)f_6(x+y-\sqrt{2}z). \]
The phase $x^2y$ is nondegenerate with respect to the projections $(x,y)\mapsto$ $x$, $y$, and $x+y$. Thus by the same argument as in the proof of Theorem \ref{main2}, the $L^2$ theorem of CLTT gives constants $C>0$ and $\epsilon>0$ so that
\[|\Lambda_{2,\lambda}(\vec{f})|\le C(1+|\lambda|)^{-\epsilon}\|f_1\|_2\|f_2\|_\infty\|f_3\|_2\|f_4\|_\infty\|f_5\|_2\|f_6\|_\infty\|f_7\|_\infty\|f_8\|_\infty\|f_9\|_\infty , \]
which holds for all $\lambda\in\R$ and $f_1,f_3,f_5\in L^2$ and $f_2,f_4,f_6,f_7,f_8,f_9\in L^\infty$.

\begin{example} \label{flex2} Multiple grouping approaches apply. 
\end{example} 
\noindent In this example, it is more optimal (i.e. leads to more $L^2$ than $L^\infty$ bounds) to fix one variable. 
Consider the functional    
\begin{align*}  
&\Lambda_{3,\lambda}(\vec{f}):=\int_{\R^4} e^{i\lambda x^2y}f_1(y+z)f_2(y-z)f_3(x+z)f_4(x-z)f_5(x+y+\sqrt{2}z)\\
&\qquad\qquad\qquad\qquad\cdot f_6(x+y-\sqrt{2}z)f_7(x+w)f_8(x-w) \eta(x,y,z,w)dxdydzdw. \end{align*} 
Theorems from CLTT do not gives $\lambda$-decay because there are too many factors ($8\not<2\cdot4$) and $x^2y$ is not simply nondegenerate with respect to the projection maps.  

Fix $z$. Then 
\begin{align*} 
\Lambda_{3,\lambda}(\vec{f}):= \int_{\R}\left( \int_{\R^3}e^{i\lambda x^2y}F_1^z(y)F_2^z(x)F_3^z(x+y)f_7(x+w)f_8(x-w)\eta(x,y,z,w)dxdydw\right)dz. 
\end{align*} 
where $F_1^z(y)=f_1(y+z)f_2(y-z)$, $F_2^z(x)=f_3(x+z)f_4(x-z)$, and $F_3^z(x)=f_5(x+\sqrt{2}z)f_6(x-\sqrt{2}z)$. The phase $x^2y$ is nondegenerate with respect to the projections $(x,y,w)\mapsto $ $x$, $y$, $x+y$, $x+w$, and $x-w$ since the operator $\partial_x\partial_y(\partial_x-\partial_y)$ annihilates the projections but not the phase. There are also $5<2\cdot 3$ factors. This leads to a bound of 
\[ |\Lambda_2(\vec{f})|\le C(1+|\lambda|)^{-\epsilon}\|f_1\|_\infty\|f_2\|_2\|f_3\|_\infty\|f_4\|_2\|f_5\|_\infty\|f_6\|_2\|f_7\|_2\|f_8\|_2,\]
where the constants $C,\epsilon>0$ are independent of the $f_j$. 

Now if we fix $z$ and $w$, 
\begin{align*} 
\Lambda_{3,\lambda}(\vec{f}):= \int_{\R^2}\left( \int_{\R^2}e^{i\lambda x^2y}F_1(x,z,w)F_2(y,z,w)F_3^z(x+y,z,w)\eta(x,y,z,w)dxdy\right)dwdz. 
\end{align*} 
where $F_1(x,z,w)=f_3(x+z)f_4(x-z)f_7(x+w)f_8(x-w)$, $F_2(x,z,w)=f_1(x+z)f_2(x-z)$, and $F_3(x,z,w)=f_5(x+z)f_6(x-z)$. The phase $x^2y$ is nondegenerate with respect to the projections $(x,y)\mapsto $ $x$, $y$, and $x+y$ since the operator $\partial_x\partial_y(\partial_x-\partial_y)$ annihilates the projections but not the phase. There are also $3<2\cdot 2$ factors, which leads to the bound
\[|\Lambda_4(\vec{f})|\le C(1+|\lambda|)^{-\epsilon}\|f_1\|_\infty\|f_2\|_2\|f_3\|_\infty\|f_4\|_2\|f_5\|_\infty\|f_6\|_2\|f_7\|_\infty\|f_8\|_\infty \]
where the constants $C,\epsilon>0$ are independent of the $f_j$. Since there are fewer $L^2$ norms, this is a weaker bound than we obtained when only $z$ was fixed.

\section{The case $P(x,y,z)=P(x)$ with deg $P\ge 3$} \label{section: P(x)}
In this section, we prove Theorem \ref{x} and discuss examples. 

\begin{proof}[Proof of Theorem \ref{x}] 
Manipulate the functional
\begin{align*}
&\Lambda_\lambda(\vec{f})= \int_{\R^3}e^{i\lambda P(x)}f_1(y+z)f_2(y-z)f_3(x+z)f_4(x-z)\\
&\qquad\qquad\qquad\qquad \cdot f_5(x+y+\sqrt{2}z)f_6(x+y-\sqrt{2}z)\eta(x,y,z)dxdydz \\
&= \int_{\R^2}f_1(y+z)\int_{\R}e^{i\lambda P(x)}f_2(y-z)f_3(x+z)f_4(x-z) \\
&\qquad\qquad\qquad\qquad \cdot f_5(x+y+\sqrt{2}z)f_6(x+y-\sqrt{2}z)\eta(x,y,z)dxdydz \\
&=: \int_{\R^2}f_1(y+z)T_\lambda(f_2,f_3,f_4,f_5,f_6)(y,z)dydz . 
\end{align*}
Since $\eta(x,y,z)$ has compact support, the integral in the final line above is equal to integrating over $(y,z)$ in a fixed compact set. Our goal is to bound $T_\lambda: L^2\times L^\infty\times L^\infty\times L^\infty\times L^\infty\to L^2$. Analyze the quantity
\[ \|T_\lambda(f_2,f_3,f_4,f_5,f_6)\|_2^2=\int_{\R^2}|T_\lambda(f_2,f_3,f_4,f_5,f_6)(y,z)|^2dydz\]
which equals
\begin{align*}
    \int_{\R^2}\int_{\R^2}e^{i\lambda (P(x)-P(x_0))} \prod_{j=2}^6f_j(v_j\cdot(x,y,z))\overline{f_j}(x_0,y,z)\eta(x,y,z)\overline{\eta}(x_0,y,z)dxdx_0dydz.
\end{align*}
Make the change of variables $(x,x_0)=(x,x+\z)$: 
\begin{align*}
&\int_{\R^4}e^{i\lambda (P(x)-P(x+\z))} \prod_{j=3}^6f_j(v_j\cdot(x,y,z))\overline{f_j}(v_j\cdot(x+\z,y,z))\eta(x,y,z)\overline{\eta}(x+\z,y,z)dxd\z dydz    \\
&= \int_{\R}\int_{\R^3}e^{i\lambda (P(x)-P(x+\z))} \prod_{j=2}^6f_j(v_j\cdot(x,y,z))\overline{f_j}(v_j\cdot(x,y,z)+v_j^1\z)\tilde{\eta}_\z(x,y,z) dxdydz d\z  \\
&= \int_{\R}\left(\int_{\R^3}e^{i\lambda (P(x)-P(x+\z))} \prod_{j=2}^6F^\z(v_j\cdot(x,y,z)) \tilde{\eta}_\z(x,y,z)dxdydz\right) d\z
\end{align*}
where $F_j^\z(v_j\cdot(x,y,z))=f_j(v_j\cdot(x,y,z))\overline{f_j}(v_j\cdot(x,y,z)+v_j^1\z)$ and $\tilde{\eta}_\z(x,y,z)=\eta(x,y,z)\overline{\eta}(x+\z,y,z)$. The integrand is supported for $\z$ in a compact set $B$, so it suffices to bound 
\[ \int_{B}\left|\int_{\R^3}e^{i\lambda (P(x)-P(x+\z))} \prod_{j=2}^6F^\z(v_j\cdot(x,y,z))\tilde{\eta}_\z(x,y,z) dxdydz\right| d\z. \]
It suffices to consider $|\lambda|\ge 1$. Let $\rho\in(0,1)$ be a parameter to be chosen later. First consider the integrand above over the set where $\z\in B$ and $|\z|\le \rho$. Then 
\begin{align*}  
& \int_{\{\z\in B:|\z|\le \rho\}}\left|\int_{\R^3}e^{i\lambda (P(x)-P(x+\z))} \prod_{j=2}^6F_j^\z(v_j\cdot(x,y,z)\tilde{\eta}_\z(x,y,z) dxdydz\right| d\z \\
&\le C \|f_3\|_\infty^2\|f_4\|_\infty^2\|f_5\|_\infty^2\|f_6\|_\infty^2\int_{\{\z\in B:|\z|\le \rho\}} \|F_2^\z\|_2d\z  \\
&\le \tilde{C}\|f_2\|_2^2\|f_3\|_\infty^2 \|f_4\|_\infty^2\|f_5\|_\infty^2\|f_6\|_\infty^2\rho^{1/2}.
\end{align*} 

Now consider the remaining $\z\in B$, i.e. those that satisfy $|\z|\ge \rho$. Then $P(x)-P(x+\z)$ is uniformly nondegenerate (for all $|\z|\ge \rho$) with respect to the projections from $(x,y,z)$ to $y-z$, $x+z$, $x-z$, $x+y+\sqrt{2}z$, $x+y-\sqrt{2}z$. This is because if
\begin{align*}  
\inf_{M\ge |\z|\ge 1}\inf_{p_j}\|P(x)-P(x+\z)-\sum_{j=2}^6p_j(v_j\cdot(x,z,z))\|=0, 
\end{align*}
then there are polynomials $p_j$ of the same degree as $P$ such that
\[ P(x)-P(x+\z)=p_2(y-z)-p_3(x+z)-p_4(x-z)-p_5(x+y+\sqrt{2}z)-p_6(x+y-\sqrt{2}z) .\]
Clearly $\square=\partial_x^2+\partial_y^2-\partial_z^2$ annihilates the right-hand side. Using the hypothesis that $\deg P\ge 3$, $\square P(x)-\square P(x+\z)=P''(x)-P''(x+\z)\not\equiv 0$, which contradicts the above displayed equality. This means that for $\z\in B$ satisfying $|\z|\ge \rho$, $|\rho|^{-1}( P(x)-P(x+\z))$ is uniformly nondegenerate. Since $5<2\cdot 3$, the $L^2$ theorem from CLTT gives the $\lambda$-decay

\begin{align*}  
&\int_{\{\z\in B:|\z|> \rho\}}\left|\int_{\R^3}e^{i\lambda (P(x)-P(x+\z))} \prod_{j=2}^6F_j^\z(v_j\cdot(x,y,z)\tilde{\eta}_\z(x,y,z) dxdydz\right| d\z \\
&= \int_{\{\z\in B:|\z|> \rho\}}\left|\int_{\R^3}e^{i(|\rho|\lambda |\rho|^{-1}(P(x)-P(x+\z))}F_2^\z(y-z)F_3^\z(x+z)F_4^\z(x-z)\right.\\
&\qquad\qquad\qquad\qquad\qquad\qquad\left.\cdot F_5^\z(x+y+\sqrt{2}z)F_6^\z(x+y-\sqrt{2}z)\tilde{\eta}_\z(x,y,z) dxdydz\right| d\z \\
&\le  \int_{\{\z\in B:|\z|> \rho\}}C(1+|\lambda||\rho|)^{-\epsilon}\|F_2^\z\|_2\|F_3^\z\|_2\|F_4^\z\|_2\|F_5^\z\|_2\|F_6^\z\|_2 d\z \\
&\le  \int_{\{\z\in B:|\z|> \rho\}}\tilde{C}(1+|\lambda||\rho|)^{-\epsilon}\|F_2^\z\|_2\|f_3\|_\infty^2\|f_4\|_\infty^2\|f_5\|^2_\infty\|f_6\|_\infty^2 d\z \\
&\le  \tilde{C}(1+|\lambda||\rho|)^{-\epsilon}\|f_3\|_\infty^2\|f_4\|_\infty^2\|f_5\|_\infty^2\|f_6\|_\infty^2 |B|^{1/2}\left(\int_{\R}\|F_2^\z\|_2^2d\z\right)^{1/2}\\
&=  \tilde{C}(1+|\lambda||\rho|)^{-\epsilon}\|f_2\|_2^2\|f_3\|_\infty^2\|f_4\|_\infty^2 \|f_5\|_\infty^2\|f_6\|_\infty^2.
\end{align*} 
Putting the above bounds together, we obtain
\begin{align*}
&\int_{B}\left|\int_{\R^3}e^{i\lambda (P(x)-P(x+\z))} \prod_{j=2}^6F^\z(v_j\cdot(x,y,z)\tilde{\eta}_\z(x,y,z) dxdydz\right| d\z\\
&\le C[(1+|\lambda||\rho|)^{-\epsilon}+\rho^{1/2}]\|f_2\|_\infty^2\|f_3\|_2^2\|f_4\|_\infty^2\|f_5\|_\infty^2\|f_6\|_\infty^2 .
\end{align*}
Choose $\rho=|\lambda|^{-1/2}$. We have proved that $T_\lambda$ is bounded with $\lambda$-decay from $L^2\times L^\infty\times L^\infty\times L^\infty\times L^\infty\to L^2$. Using this result in the functional $\Lambda_\lambda$ we manipulated at the beginning of the proof, conclude
\begin{align*}
&|\Lambda(\vec{f})|= \left|\int_{\R^3} e^{i\lambda P(x)}f_1(y+z)f_2(y-z)f_3(x+z)f_4(x-z)\right.\\
&\qquad\qquad\qquad\qquad\left.\cdot f_5(x+y+\sqrt{2}z)f_6(x+y-\sqrt{2}z)\eta(x,y,z)dxdydz \right|\\
&\qquad\qquad = \left|\int_{S}\int_{\R} f_1(y+z)T_\lambda(f_2,f_3,f_4,f_5,f_6)(y,z)dydz\right|\\
&\qquad\qquad \le \left(\int_S\int_{\R}|f_1(y+z)|^2dydz\right)^{1/2}\|T_\lambda(f_2,f_3,f_4,f_5,f_6)\|_2 \\
&\qquad\qquad \le \|f_1\|_2|S|^{1/2}C(1+|\lambda|)^{-\tilde{\epsilon}}\|f_2\|_2\|f_3\|_\infty\|f_4\|_\infty\|f_5\|_\infty\|f_6\|_\infty ,
\end{align*}
where $\tilde{\epsilon}=\min(\epsilon,1)/4$ and $S$ is a finite-measure set with the property that $\text{supp } \eta\subset \R^2\times S$. 
This proves Theorem \ref{x}.
\end{proof} 

\begin{remark}\label{rem} The above argument proving $\lambda$-decay for the special case $P(x)$ works for any polynomial phase $P(x,y,z)$ of degree $d\ge 3$ for which the quantity
\[ \inf_{M\ge |\z|\ge 1}\inf_{p_j}\|P(x,y,z)-P(x+\z,y,z)-\sum_{j=2}^6p_j(v_j\cdot(x,y,z))\|_d \]
is nonzero. We describe in the following section some examples of phases $P(x,y,z)$ for which the grouping theorem (Theorem \ref{main2}) applies but for which this displayed nondegeneracy quantity vanishes. 
\end{remark}

\begin{remark}Also observe that any phases $P(x)$ for which $\lambda$-decay is proved imply corresponding $\lambda$-decay results for the phase $P(y)$ by performing the change of variable switching $x$ with $y$. 
\end{remark}

Now we demonstrate more examples.
\begin{example} \label{2not1} $P(x,y,z)=x^3$: Theorem \ref{x} applies, but not Theorem \ref{main2}. \end{example} 
\noindent The hypothesis of Theorem \ref{main2} is not satisfied since $\partial_y\partial_x(\partial_x-\partial_y)x^3=0$. However, $x^3$ is a polynomial in $x$ of degree at least 3, so by Theorem \ref{x}, there exist $C>0$ and $\epsilon>0$ so that
\[ |\Lambda_\lambda(x^3,\vec{f})|\le C(1+|\lambda|)^{-\epsilon}\|f_1\|_2\|f_2\|_2\|f_3\|_\infty\|f_4\|_\infty\|f_5\|_\infty\|f_6\|_\infty \]
for all $\lambda\in\R$ and all $f_1,f_2\in L^2$ and $f_3,f_4,f_5,f_6\in L^\infty$.

\begin{example}\label{both} $P(x,y,z)=x^2y^2$: Theorems \ref{main2} and \ref{x} apply. 
\end{example}
\noindent The hypothesis 
\[ \inf_{z\in\R}\sup_{|(x,y)|\le1}|\partial_x\partial_y(\partial_x-\partial_y)x^2y^2|>0 \]
of Theorem \ref{main2} is satisfied. Thus there exists $C>0$ and $\epsilon>0$ such that 
\[ |\Lambda_\lambda(x^2y^2,\vec{f})|\le C(1+|\lambda|)^{-\epsilon}\|f_1\|_2\|f_2\|_\infty\|f_3\|_2\|f_4\|_\infty\|f_5\|_2\|f_6\|_\infty \]
for all $\lambda\in\R$ and $f_1,f_3,f_5\in L^2$ and $f_2,f_4,f_6\in L^\infty$. By the Remark \ref{rem}, we may use the argument in the proof of Theorem \ref{x} if we check that  
\[ \inf_{M\ge |\z|\ge 1}\inf_{p_j}\|P(x,y,z)-P(x+\z,y,z)-\sum_{j=2}^6p_j(v_j\cdot(x,y,z))\|\not=0,\]
where $v_2=(0,1,-1)$, $v_3=(1,0,1)$, $v_4=(1,0,-1)$, $v_5=(1,1,\sqrt{2})$, and $v_6=(1,1,-\sqrt{2})$. If the above quantity is 0, then there are polynomials $p_j:\R\to\R$ and $\z\in [1,M]$ such that 
\[ x^2y^2-(x+\z)^2y^2=p_2(y-z)+p_3(x+z)+p_4(x-z)+p_5(x+y+\sqrt{2}z)+p_6(x+y-\sqrt{2}z). \]
But if we apply $(\partial_y+\partial_z)\partial_y(\partial_x-\partial_y)$ to both sides, the right-hand side is annihilated and the left-hand  side is $-4\z$, which is a contradiction. Thus an argument analogous to the proof of Theorem \ref{x} gives constants $C>0$ and $\epsilon>0$ such that
\[  |\Lambda_\lambda(x^2y^2,\vec{f})|\le C(1+|\lambda|)^{-\epsilon}\|f_1\|_2\|f_2\|_2\|f_3\|_\infty\|f_4\|_\infty\|f_5\|_\infty\|f_6\|_\infty \]
for all $\lambda\in\R$ and all $f_1,f_2\in L^2$ and $f_3,f_4,f_5,f_6\in L^\infty$.


\end{document}